\documentclass [11pt,a4paper]{article}
\usepackage{amsmath,amsthm,amssymb,esint}
\usepackage{epsfig,graphics}

\newcommand{\Ndash}{\nobreakdash--}
\newtheorem{lem}{Lemma}
\newtheorem{lemma}{Lemma}[section]

\newtheorem{theor}{Theorem}

\DeclareRobustCommand{\qedif}
{\hbox{}\nobreak\quad\eqno\hbox{\qedsymbol}}

\numberwithin{equation}{section}

\title{Divided Differences, Square Functions and a Law of the Iterated Logarithm}

\author{Artur Nicolau\\\small Departament de Matem\`{a}tiques\\\small
Universitat Aut\`{o}noma de Barcelona\\\small 08193
 Barcelona, Spain\\\small artur@mat.uab.es}

\date{}

\begin{document}

\maketitle

\let\thefootnote\relax\footnote{The author is supported in part by the grants MTM2008-00145 and
2009SGR420}

\section{Introduction}\label{section1}

Let $f$ be a real valued measurable function defined in an open
set~${\cal U}\subset \mathbb{R}^d$. If $x,t\in\mathbb{R}^d$ satisfy
$x-t, x, x+t\in {\cal U}$, we consider the (symmetric) divided
difference $\Delta (f) (x,t)$ and the second (symmteric) divided
difference $\Delta_2 (f) (x,t)$ defined as
\begin{align*}
\Delta (f) (x,t)&= \frac{f(x+t)-f(x-t)}{2|t|},\\*[5pt] \Delta_2
 (f) (x,t)&= \frac{f(x+t)+f(x-t)-2f(x)}{2|t|}.
\end{align*}
It is well known that differentiability properties of the
function~$f$ can be described by size conditions on the
differences~$\Delta_2f$. Actually for $\delta>0$ consider the square
function
$$
g_{\delta}^2(f)(x)=\int_{\|t\|<\delta} \Delta_2^2 (f) (x,t)
\frac{dm(t)}{|t|^d},\quad x\in {\cal U},
$$
where $dm(t)$ denotes Lebesgue measure in $\mathbb{R}^d$.  We denote $g(f) = g_1 (f)$. A
classical result by Stein and Zygmund, extending previous work by
Marcinkiewicz and Zygmund, says that the set of points in ${\cal U}$
where $f$ is differentiable and the set of points $x\in {\cal U}$
for which there exists $\delta=\delta(x)>0$ such that
$g_{\delta}(f)(x)<\infty$ and $\sup \{ |\Delta_2 f(x,h) | :\|h\|
<\delta\}<\infty$, can differ at most by a set of Lebesgue measure
zero. See \cite{SZ2} or \cite[p.~262]{St1}.

In this work we study the growth of the divided differences of a
function at the points where the function is not differentiable. In
the one dimensional case, under certain assumptions on the function,
Anderson and Pitt obtained very nice results in their
paper~\cite{AP}. For instance they considered the Zygmund class of
continuous one variable functions~$f$ for which $\|f\|_*=\sup \{
|\Delta_2 (f) (x,h)| : x, h\in\mathbb{R} \}<\infty$. Since $|\Delta
(f) (x+h,h)-\Delta (f) (x+h/2,h/2)|\le |\Delta_2 f(x+h,h)|$, for any
$x,h\in\mathbb{R}$, iterating one obtains
$$
\left|\frac{f(x+h)-f(x)}{h}\right|\le \|f\|_* \ln_2 (1/h) + 2|f(x +
2^N h) - f(x)| \, , x \in \mathbb{R} \, , 0 < h < 1/2
$$
where $N$ is the integer such that $1/2 < 2^N h < 1$. Hence for any
$x \in \mathbb{R}$, the growth of the divided differences $|\Delta
(f) (x,h)|$ is at most proportional to $\ln (1/h)$ for $0<h<1/2$.
Moreover this uniform estimate is sharp. However, Anderson and Pitt
proved the following pointwise estimate which is a version of
Kolmogorov's Law of the Iterated Logarithm and improves the previous
trivial estimate. At almost every point  $x\in\mathbb{R}$, one has
\begin{equation}\label{eq1.1}
\limsup_{h\to 0} \frac{|f(x+h)-f(x)|} {|h|\sqrt{\ln 1/|h| \ln\ln\ln
1/|h|}} \leq C \|f\|_* ,
\end{equation}
where $C$ is a universal constant. The result is sharp. For instance, fixed $b>1$, the Weierstrass-Hardy lacunary series
$$
f_b(x)=\sum^{\infty}_{n=1} b^{-n} \cos (b^nx),\quad x\in\mathbb{R}
$$
is in the Zygmund class and there exists a constant~$C_1 = C_1 (b)$
such that the $\limsup$ in~\eqref{eq1.1} is bigger than~$C_1$ at
almost every  $x\in\mathbb{R}$. See \cite{W}. Differentiability of
functions in the Zygmund class has been studied in \cite{Ma},
\cite{DLlN1} and \cite{DLlN2}. The result of Anderson and Pitt is
very nice but the assumption that $f$ is in the Zygmund class is
somewhat unnatural. Also, instead of estimating the divided
differences of a function by a logarithm of the scale, one expects
to estimate them by truncated versions of convenient square
functions. This is what happens when studying boundary behavior of
harmonic functions in the upper-half space. Let $u$ be a harmonic
function in an upper half space and let $A(u)$ be its Lusin area
function. Classical results of Calder\'on, Zygmund and Stein tell
that the set of points where $u$ has non-tangential limit and the
set of points where $A(u)$ is finite, can differ at most by a set of
Lebesgue measure~$0$.  See for instance \cite[p.~206]{St1} or
\cite[p.~43]{BM}. On the complement of this set, the growth of~$u$
is controlled by a truncated variant of $A(u)$ via a convenient
version of the Law of the Iterated Logarithm. See \cite{BKM1},
\cite{BKM2} or \cite[p.~65]{BM}.

Let us first restrict attention to the one dimensional case. Let
${\cal U}$ be an open set of the real line $\mathbb{R}$ and let $f
\in L^2_ {\mathrm{loc}}({\cal U})$. Given $x \in {\cal U}$ consider
$h_0 = h_0 (x) = \min \{1, \text{dist} (x, \mathbb{R} \setminus
{\cal U}) / 2\}$.  Instead of the \emph{vertical} square
function~$g(f)$, consider the \emph{conical} square function~$A(f)$
defined as
$$
A^2 (f) (x)=\int_{\Gamma(x)} \Delta_2^2(f)(s,t)\frac{ds\,dt}{t^2},\quad x\in\mathbb{R},
$$
where $\Gamma(x)=\{(s,t)\in\mathbb{R}^2_+:|s-x|<t<h_0\}$ is the cone
centered at~$x$ of height $h_0$. In contrast with (1.1), we do not
want to assume any kind of regularity on the function $f$. Since the
 behavior of the divided differences of a function $f$ may change completely if one
changes the definition of $f$ in a set of Lebesgue measure zero, one
can not expect to control the divided differences by an square
function as~$A(f)$ or $g(f)$. However, it turns out that means of
divided differences defined as
$$
\tilde{\Delta}(f)(x,h)=\int^h_{h/2} \int^{x+t}_{x-t} \Delta (f)
(s,t)\frac{ds\,dt}{2 t^2} \quad , \,  x \in \mathbb{R}, \,  0<h<1 ,
$$
can be controlled by truncated versions of $A(f)$ defined as
$$
A^2(f)(x, h)=\int_{\Gamma(x)\cap \{t\ge h\}} \Delta_2^2
(f)(s,t)\frac{ds\,dt}{t^2},\quad x\in\mathbb{R},\quad 0<h<1.
$$

\begin{theor}\label{theor1}
Let $f\in L^2_ {\mathrm{loc}}({\cal U})$. Then at almost every
point~$x \in \{x \in {\cal U} : A(f)(x) = \infty \}$, one has
$$
\limsup_{h\to 0} \frac{| \tilde{\Delta} (f)(x,h)|}{\sqrt{A^2(f)(x,
h)\ln\ln A^2(f)(x, h)}} \le  \sqrt{2 \ln 2}.
$$
\end{theor}

The result is sharp up to the value $ \sqrt{2 \ln 2}$ in the sense
that when $f=f_b$ is the Hardy-Weierstrass lacunary series mentioned
above, the $\limsup$ in the statement is bounded below at almost
every point $x \in \mathbb R $. Let $f$ be a function in the Zygmund
class. Since there exists an absolute constant $C>0$ such that
$|\Delta (f)(x,h)- \tilde{\Delta}(f)(x,h)| \leq C \|f\|_* $ and
$A^2(f)(x, h) \leq C \|f\|_{*}^2 \ln (1/h)$, the estimate (1.1) of
Anderson and Pitt follows from Theorem 1. It is worth mentioning
that we do not know if the analogue of Theorem 1 holds when one
replaces $A(f)(x,h)$ by a truncated version of $g_1 (f)$. An
analogue situation occurs when studying the growth of a harmonic
function in an upper half space outside its Fatou set. As mentioned
above, Ba\~nuelos, Klemes and Moore proved a version of the Law of
the Iterated Logarithm which controls the growth of the harmonic
function in terms of the size of its truncated area function. See
\cite{BKM1} or \cite[p.~65]{BM}. However a similar result replacing
the {\it conical} Lusin area function by the {\it vertical}
Littlewood-Payley function is not known. See
  \cite[p.~114]{BM}.

The main technical step in the proof of our result is the following
good $\lambda$-inequality with provides the right subgaussian decay:
there exists a universal constant $C>0 $ such that for any  $f\in
L^2  ([0,1])$ and any numbers $N,M>0$, one has
\begin{equation}\label{eq1.2}
|\{x \in [0,1]: \sup_{1 \ge y \ge h} (\tilde{\Delta} (f)(x,y)-\tilde
{\Delta} (f)(x, 1))\ge M ;   A^2(f)(x,h)\le N \}| \le C \exp \left(
- M^2 / CN \right)
\end{equation}
Theorem~\ref{theor1} follows from this subgaussian estimate by
standard arguments. Subgaussian estimates in different contexts in
analysis can be founded in \cite{CWW}, \cite{BKM1}, \cite{BM},
\cite{Ma} and \cite{SV}. Our proof of \eqref{eq1.2} is organized in
two steps. First we state and prove a dyadic version of
\eqref{eq1.2} and later we use an averaging procedure due to
J.~Garnett and P.~Jones~(\cite{GJ}) to transfer the result in the
dyadic setting to the continuous one.

The square function~$A(f)$ can also be used as a substitute
of~$g(f)$ in the classical result of Stein and Zygmund mentioned
above. More concretely the following analogue of this classical
result holds.

\begin{theor}\label{theor2}
Let $f$ be a measurable function defined in an  open set ${\cal
U}\subset\mathbb{R}$. Consider the set $A=\{x\in {\cal U}: f\text{
is differentiable at }x\}$ and the set~$B$ of points $x\in {\cal U}$
for which there exists $\delta=\delta (x)>0$ such that
$\sup\{|\Delta_2 (f) (x,h)|: |h |<\delta\}<\infty$ and
$$
\int_{\Gamma(x)\cap \{0<t<\delta\}} \Delta_2^2 (f) (s,t)
\frac{ds\,dt}{t^2}<\infty.
$$
Then, the sets~$A$ and $B$ can differ at most by a set of Lebesgue measure zero.
\end{theor}

Observe that if we change the function~$f$ at a set of Lebesgue
measure zero, the set of points where $f$ is differentiable may
change completely but the square function $A(f)$ remains unchanged.
So, the condition $\sup |\Delta_2 (f) (x,h)|<\infty$ in the set~$B$
is really needed.


For $1<p < \infty$ let $W^{1,p} (\mathbb{R}^d)$ be the Sobolev space
of functions in $L^p(\mathbb{R}^d)$ whose  partial derivatives, in
the sense of distributions, are in $L^p(\mathbb{R}^d)$. If $2d/(d+1)
< p < \infty$, a function  $f\in L^p (\mathbb{R}^d)$ is in the
Sobolev space~$W^{1,p} (\mathbb{R}^d)$ if and only if $g_1(f)\in L^p
(\mathbb{R}^d)$. See \cite[p.~163]{St1}. Note that when $d=1$, the
result holds for any $1<p<\infty$. A similar result holds in our
setting.

\begin{theor}\label{theor3}
Let $1<p<\infty$. A function~$f\in L^p(\mathbb{R})$ is in the
Sobolev space~$W^{1,p} (\mathbb{R}) $ if and only if $A(f)\in L^p(\mathbb{R})$.
Moreover, there exists a constant~$C=C(p)>0$ such that
$C^{-1}\|A(f)\|_p\le \|f'\|_p\le C\|A(f)\|_p$ for any $f\in
W^{1,p} (\mathbb{R})$.
\end{theor}


 Let us now explain our results in higher dimensions. We start recalling some classical results.  Rademacher's
Theorem says that a Lipschitz function defined in an open set of
$\mathbb{R}^d $ is differentiable at almost every point of the open
set. A classical refinement due to Stepanov says that a measurable
function $f$ defined in an open set ${\cal U} \subset \mathbb{R}^d $
is differentiable at almost every point of the set
$$
\{x \in {\cal U} : \limsup_{|h| \to 0 }  \frac{|f(x+h) - f(x)|}{|h|}
< \infty \}
$$
See  \cite[p.~250]{St}. Stepanov also constructed a continuous
nowhere differentiable function in $\mathbb{R}^2$ whose ordinary
partial derivatives exist at almost every point.  Fixed $x \in
\mathbb{R}^d$ and $\varepsilon>0$, consider the condition
\begin{equation}\label{eq6.2}
\sup_{h \in \mathbb{R}^d : |h| < \varepsilon  } |\Delta_2 (f) (x, h)
| < \infty
\end{equation}
This condition is certainly satisfied if $f$ is differentiable at
the point $x$, but, as mentioned before, the converse is far from
being true. It turns out that \eqref{eq6.2} plays the role of a
Tauberian condition allowing one to deduce differentiability from
existence of partial derivatives.  This is the content of next
result which may have independent interest.  It is analogue to a
classical result by Stein and Zygmund where under the assumption
\eqref{eq6.2}, one deduces ordinary differentiability at almost
every point where differentiability in the harmonic sense holds. See
\cite[p.260]{St}.

\begin{lem}\label{lem4}
Let $\{e_i : i = 1,2, \ldots, d \}$ be the canonical basis of
$\mathbb{R}^d$. Let $f$ be a measurable function defined in an open
set  ${\cal U} \subset \mathbb{R}^d $. Then $f$ is differentiable at
almost every point $x \in {\cal U }$ where the following two
conditions hold
$$
\limsup_{t \in \mathbb{R}, t \to 0} | \frac{f(x+ t e_i) - f(x)}{t} |
< \infty \, , i =1, \ldots , d
$$
and
$$
\limsup_{h \in \mathbb{R}^d ,  |h| \to 0} |\Delta_2 (f) (x, h)| <
\infty
$$
\end{lem}

 Let ${\cal U}$ be an open set in the euclidean space $\mathbb{R}^d$. Let
$f \in L^2_ {\mathrm{loc}}({\cal U})$. Given $\xi \in \mathbb{R}^d$
with $|\xi|= 1$, for $x \in {\cal U}$ and $0< t < h_0 = \min \{ 1,
\text{dist} (x, \mathbb{R}^d  \setminus  {\cal U}) / 2  \}$,
consider the divided difference and the second divided difference in
the direction of $\xi$ given by $ \Delta_\xi (f) (x,t) =
 (f( x
+t \xi)-f(x  -t \xi))/ 2t $ and $ \Delta_{2, \xi} (f) (x,t) = (f(x +
t \xi) + f (x - t \xi ) - 2 f(x ) )/ 2t$. For  $x \in {\cal U}$ and
$0< h < h_0 $, the mean divided difference of $f$ in the direction
$\xi$ is defined as
$$
\tilde {\Delta}_{\xi} (f)(x,h) =\int^h_{h/2} \int^{t}_{-t}
\Delta_\xi (f) (x+ s \xi ,t)  \frac{ds dt}{2 t^2}
$$
and the square function in the direction $\xi$  is defined as
$$
A_{\xi}^2(f)(x,h)=\int_h^{h_0} \int_{-t}^{t}   \Delta_{2, \xi}^2 (f)
(x+ s \xi,t) \, \frac{ds dt}{t^2}
$$
Note that both $\tilde {\Delta}_{\xi} (f)(x,h)$ and
$A_{\xi}^2(f)(x,h)$ are defined at almost every point $x \in {\cal
U}$. As before, we denote $A_{\xi} (f)(x)= A_{\xi}(f)(x,0)$. Our one
dimensional results easily give the following statement.
\begin{theor}\label{theor4}
 Let ${\cal U}$ be an open subset of $\mathbb{R}^d$. Fix $\xi
\in \mathbb{R}^d$ with $|\xi|= 1$.

(a) Let $f$ be a measurable function defined in ${\cal U}$. Consider
the set $A$ of points in ${\cal U}$ on which $f$ has directional
derivative in the direction of $ \xi $ and the set~$B$ of points
$x\in {\cal U}$ for which there exists $\delta=\delta (x)>0$ such
that $\sup\{|\Delta_2 (f) (x,h)|: |h |<\delta\}<\infty$ and
$$
\int_0^{\delta} \int_{-t}^{t}   \Delta_{2, \xi}^2 (f) (x+ s \xi,t)
\, \frac{ds dt}{t^2}<\infty.
$$
Then, the sets~$A$ and $B$ can differ at most by a set of Lebesgue
measure zero.



(b ) Assume $f \in L^2_ {\mathrm{loc}}({\cal U})$. At almost every
point~$x \in \{ x \in {\cal U} : A_\xi (f) (x) = \infty \}$, one has
$$
\limsup_{h\to 0} \frac{ |\tilde{\Delta}_\xi  (f)(x,h)
|}{\sqrt{A_\xi^2(f)(x, h)\ln\ln A_\xi^2(f)(x, h)}} \le  \sqrt{2 \ln
2}.
$$

(c) Let $1< p < \infty$ and assume $f \in L^p (\mathbb{R}^d)$. Then
the directional derivative in the sense of distributions $D_{\xi} f$
is a function in $   L^p (\mathbb{R}^d)$  if and only if $A_\xi (f)
\in  L^p (\mathbb{R}^d)$. Moreover there exists a constant $C=C(p,
d) >0$ independent of $f$ and $\xi$ such that $C^{-1}\|A_{\xi}
(f)\|_p\le \|D_\xi (f) \|_p\le C\|A_{\xi} (f)\|_p$ for any $f \in
L^p (\mathbb{R}^d)$ such that $D_{\xi}  f \in  L^p (\mathbb{R}^d)$.
\end{theor}

Given distinct points $\xi_1 , \ldots, \xi_d $ in the unit sphere of
$\mathbb{R}^d$ , consider
$$
{\cal A}_\delta^2 (f) (x) = \sum_{i=1}^d \int_0^{\delta}
\int_{-t}^{t} \Delta_{2, \xi_i}^2 (f) (x+ s \xi_i,t) \, \frac{ds
dt}{t^2} \, , x \in \mathbb{R}^d \, .
$$
So ${\cal A}_{h_0}^2 (f) = \sum_{i} A_{\xi_i}^2(f)$. From Theorem 4
and Lemma 1 one easily deduces that the set of points where $f$ is
differentiable coincides up to sets of Lebesgue measure
 $0$, with the set of points $x \in {\cal U}$ for which there exists $\delta =  \delta (x) >0$  such that both conditions
  \eqref{eq6.2} and ${\cal A}_\delta^2 (f) (x) < \infty$ hold. From (c) of Theorem 4 one can easily deduce a
characterization of Sobolev spaces in several variables in terms of
the {\it conical} square function ${\cal A}$ which holds for any
$1<p<\infty$.  More concretely, if $ f \in L^p (\mathbb{R}^d) $ then
$f \in W^{1,p} (\mathbb{R}^d)$ if and only if ${\cal A}_1 (f) \in
L^p (\mathbb{R}^d) $. It would be interesting to compare this result
with the beautiful characterization of Sobolev spaces given in
\cite{AMV}.

We finally introduce another higher dimensional natural extension of
the square function $A$ which describes differentiability at almost
every point of a given set of the euclidean space. Let $f$ be a measurable function defined in an open set ${\cal U} \subset \mathbb{R}^d$.
Let $S^{d-1}$ denote the unit sphere in $\mathbb{R}^d $ and let
$\sigma$ be the normalized surface measure in $S^{d-1}$. Assume $f \in
L^2_ {\mathrm{loc}}({\cal U})$. Consider
\begin{equation}\label{eq6.1}
{\mathbf{ A}}^2 (f) (x,h) = \int_{S^{d-1}} A_{\xi}^2 (f) (x,h) d
\sigma (\xi) \, , x \in {\cal U} , 0<h<1
\end{equation}
and ${\mathbf{ A}} (f) (x)= A(f)(x,0)$. Consider also the following
averaged version of $\tilde {\Delta}_\xi$. Given a measurable subset
$E \subset S^{d-1}$, consider
$$
\tilde {\Delta} (f) (x,h,E )= \int_E \tilde {\Delta}_\xi (f)(x,h) d
\sigma (\xi) \, , x \in \mathbb{R}^d, 0 < h <1
$$

\begin{theor}\label{theor5}
(a) Let ${\cal
U}$ be an open set of $\mathbb{R}^d$ and let $f \in L^2_ {\mathrm{loc}}({\cal U})$. Consider the set $A=\{x\in {\cal U}: f\text{
is differentiable at }x\}$ and the set~$B$ of points $x\in {\cal U}$ such that ${\mathbf{ A}}  (f) (x)< \infty$
for which there exists $\delta=\delta (x)>0$ such that
$\sup\{|\Delta_2 (f) (x,h)|: |h |<\delta\}<\infty$.
Then, the sets~$A$ and $B$ can differ at most by a set of Lebesgue measure zero.

(b) There exists a constant $C=C(d)>0$ such that for any $f \in L^2_ {\mathrm{loc}}(\mathbb{R}^d )$, for almost every point~$x \in \{ x \in \mathbb{R}^d  :
\mathbf{ A} (f) (x) = \infty \}$ and for any measurable subset $E
\subset S^{d-1}$, one has
$$
\limsup_{h\to 0} \frac{ |\tilde {\Delta} (f) (x,h,E)
|}{\sqrt{{\mathbf{ A}}^2 (f) (x,h) \ln\ln {\mathbf{ A}}^2 (f)
(x,h)}} \le C.
$$
\end{theor}

Finally let us mention an easy consequence of Theorems 4 and  5
which is related to a classical result. Let $f$ be a function
defined in an open subset ${\cal U} \subset \mathbb{R}^d$. Let $w$
be a function defined in $(0,1]$ such that for any $x \in {\cal U}$
and $ h \in \mathbb{R}^d$, $0 <
 |h| \leq h_0 = \min \{ 1,
dist (x, \mathbb{R}^d  \setminus  {\cal U}) / 2  \}$, one has $
|f(x+h) + f(x-h) - 2 f(x)| \leq |h| w(|h|) $. Consider
$$
W(s) = \int_s^1 w^2 (t) \frac{dt}{t} \quad , 0<s<1 \, .
$$
If $w$ is increasing and $W(0) < \infty$, Stein and Zygmund proved
that $f$ is differentiable at almost every point of ${\cal U}$ . See
\cite{SZ2} or part (a) of Theorem 5. See also \cite{Ma} and
\cite{DN}. If $W(0 )= \infty$, part (b) of Theorem 4 gives that
there exists a constant $C_1 = C_1 (d)$ only depending on the
dimension such that for any $\xi \in \mathbb{R}^d$ with $|\xi|=1$,
at almost every $x \in {\cal U}$ one has
$$
\limsup_{h\to 0} \frac{ |\tilde{\Delta}_\xi  (f)(x,h) |}{\sqrt{ W
(|h|)\ln\ln W(|h|)}} \le C_1.
$$





The paper is organized as follows. Next Section is devoted to the discrete setting of dyadic martingales
 and to obtain the exponential inequalities relating the growth of a dyadic martingale and its quadratic variation.
 In Section 3 we consider the one dimensional continuous setting and obtain the subgaussian
 estimate \eqref{eq1.2}
  relating  $ \tilde{\Delta} (f)$ and $A(f)$ which is the main technical tool in the proof of Theorem 1.
  In Sections 4 and 5 we again use the results in the discrete setting to prove Theorems 2 and 3, respectively. In Section 6 we consider functions of several real variables and prove Theorems 4 and 5. Finally in Section 7 several natural questions closely related to our results are collected.

\bigskip

\section{The Discrete Setting}\label{section2}

For $1\le \rho \le 4$ and $k=0,1,2,\dotsc$, let $\mathcal{D}_k
(\rho)$ be the collection of $\rho$-dyadic intervals of
generation~$k$ in $\mathbb{R}$ of the form $[j2^{-k}\rho, (j+1)
2^{-k}\rho)$ where $j$ is an integer. Let
$\mathcal{D}(\rho)=\bigcup\limits_{k\ge 0}\mathcal{D}_k (\rho)$ be
the collection of all $\rho$-dyadic intervals. For $x\in\mathbb{R}$
let $I_k^{(\rho)}(x)$ be the unique interval in $\mathcal{D}_k
(\rho)$ which contains $x$.  Also $|E|$ denotes the Lebesgue measure
of the measurable set $E \subset \mathbb{R}$. A $\rho$-dyadic
martingale is a sequence of locally integrable
functions~$S=\{S_k^{(\rho)}\}_k$ such that for any $k=0,1,2,\dotsc$,
the function~$S_k^{(\rho)}$ is measurable with respect to the
$\sigma$-algebra ${\cal F}_k$ generated by $\mathcal{D}_k(\rho)$ and
the conditional expectation of $S_{k+1}$ respect to ${\cal F}_{k}$
is $S_{k}$. In other words, for any $k=0,1,2,\dotsc$, the
function~$S_k^{(\rho)}$ is constant in each $\rho$-dyadic interval
of $\mathcal{D}_k(\rho)$ and
$$
\int_I \left(S^{(\rho)}_{k+1}(x)-S_k^{(\rho)}(x)\right)\,dx=0
$$
for any $I\in \mathcal{D}_k(\rho)$. The truncated maximal function of the martingale~$S$ is defined by
$$
M_n(S)(x)=\sup_{k\le n} |S_k^{(\rho)}(x)|,\quad x\in\mathbb{R},\quad n=1,2,\dotsc
$$
The truncated quadratic variation of $S$ is defined by
$$
\langle S\rangle_n^2 (x)=\sum_{k=1}^n \left(S_k^{(\rho)}(x)-S_{k-1}^{(\rho)}(x)\right)^2,\quad x\in\mathbb{R},\quad n=1,2,\dotsc
$$
It is well known that many properties on the asymptotic  behavior of
a martingale can be described in terms of the size of its quadratic
variation. More concretely, the sets
$\{x\in\mathbb{R}:\lim\limits_{k\to\infty} S_k^{(\rho)}(x)\text{
exists}\}$, $\{x\in\mathbb{R}:M_{\infty} (S)(x)<\infty\}$ and
$\{x\in\mathbb{R}:\langle S\rangle_{\infty}(x)<\infty\}$ can only
differ on a set of Lebesgue measure zero. See \cite{BG1} or
\cite[p.~64]{BM}. Also, fixed $0<p<\infty$ and
$I\in\mathcal{D}(\rho)$, the maximal function~$M_{\infty}(S)$ is in
$L^p(I)$ if and only if $\langle S\rangle_{\infty}$ is in $L^p(I)$.
See \cite{BG1}, \cite{BG2}. These results give comparisons between
$M_{\infty}(S)$ and $\langle S\rangle_{\infty}$ on the sets where
they are finite. In its complement, the following Law of the
Iterated Logarithm governs the growth of the martingale,
$$
\limsup_{n\to\infty} \frac{|S_n^{(\rho)}(x)|}{\sqrt{\langle S\rangle^2_n(x)\ln\ln \langle S\rangle_n^2(x)}}\le \sqrt{2}
$$
at almost every point  $x\in\{x\in\mathbb{R}:\langle S\rangle_{\infty}(x)=\infty\}$.
This result follows from good $\lambda$-inequalities, with
subgaussian decay, which relate the growth of $M_n(S)$ and $\langle
S\rangle_n$. See \cite{St}, \cite{CWW} or \cite{BM}.  We start with a well known result in the same vein (see
\cite[p.~47]{BM}) whose proof is included for the sake of completeness.

\begin{lemma}\label{lemma2.1}
Let $S= \{S_k^{(\rho)}\}$ be a $\rho$-dyadic martingale. Fix
$I_0\in \mathcal{D}_0(\rho)$ and assume $S_0^{(\rho)}\equiv 0$ on
$I_0$. Then
$$
\int_{I_0} \exp \left(S_n^{(\rho)}(x)-\frac{1}{2}\langle  S
\rangle^2_n(x)\right)\,dx\le |I_0|,\quad n=1,2,\dotsc
$$
\end{lemma}

\begin{proof}
Fix $I\in\mathcal {D}_{n-1}(\rho)$. Denote by $a_n(I)$ the constant
value of $S_{n-1}^{(\rho)}-\frac{1}{2}\langle  S \rangle_n^2$ on
$I$. Then
$$
\int_I \exp \left(S_n^{(\rho)}(x)-\frac{1}{2}\langle S
\rangle_n^2(x)\right)\,dx\!=\! \exp \left(a_n (I) \right)
\left(\int_I\exp
\left(S_n^{(\rho)}(x)-S_{n-1}^{(\rho)}(x)\right)\,dx\right).
$$
Let $g=S_n^{(\rho)}-S_{n-1}^{(\rho)}$. Observe that $|g|$ has a
constant value on $I$ which will be called $|g(I)|$.  Since $\int_I
g(x)\,dx=0$, using the elementary estimate $\cosh (x) \leq \exp (x^2
/ 2)$, we deduce
$$
\int_I\exp (g(x))\,dx = |I|\cosh(|g(I)|)\le |I|\exp
\left(\frac{1}{2} |g(I)|^2 \right).
$$
Hence
$$
\int_I\exp \left(S_n^{(\rho)}(x)-\frac{1}{2}\langle S
\rangle_n^2(x)\right)\,dx\le \int_I\exp
\left(S_{n-1}^{(\rho)}(x)-\frac{1}{2}\langle S
\rangle_{n-1}^2(x)\right)\,dx.
$$
Adding over all $I\in\mathcal{D}_{n-1}(\rho)$ contained in $I_0$ we
deduce
$$
\int_{I_0} \exp \left(S_n^{(\rho)} (x) -\frac{1}{2}\langle
S\rangle_n^2  (x) \right)\,dx\le\int_{I_0}\exp
\left(S^{(\rho)}_{n-1} (x) -\frac{1}{2}\langle S\rangle_{n-1}^2 (x)
\right)\,dx
$$
and the result follows.
\end{proof}

We now easily deduce

\begin{lemma}\label{lemma2.2}
Let $S=\{S_k^{(\rho)}\}$ be a $\rho$-dyadic martingale. Fix
$I_0\in\mathcal{D}_0(\rho)$ and assume $S_{0}^{(\rho)}\equiv 0$ on
$I_0$. Then for any $n=1,2,\dotsc$ and any $\lambda >0$ one has
$$
\left| \left\{x\in I_0: \sup_{k\le n} \left(S_k^{(\rho)}(x)-\frac{1}{2}\langle S\rangle_k^2(x)\right)>\lambda\right\}\right|\le e^{-\lambda}|I_0|.
$$
\end{lemma}

\begin{proof}
Fix $\lambda >0$. Fix the integer $n\ge 1$ and consider the stopping
time $\tau(x)$ defined as the minimum between the indices~$k\le n$
for which $S_k(x)-\frac{1}{2}\langle S\rangle_k^2(x)>\lambda$, and
$n$. Apply Lemma~\ref{lemma2.1} to the stopped martingale~$S^{\tau}$
defined as $S^{\tau}(x)=S^{(\rho)}_{\tau(x)}(x)$, to get
$$
\int_{I_0} \exp \left(S_n^{\tau}(x)-\frac{1}{2}\langle S^{\tau}\rangle_n^2(x)\right)\,dx\le|I_0|.
$$
Since $S_n^{\tau}-\frac{1}{2}\langle S^{\tau}\rangle_n^2>\lambda$ on
the set $ \left\{x\in I_0: \sup_{k\le n}
(S_k^{(\rho)}(x)-\frac{1}{2}\langle S\rangle^2_k(x) ) >\lambda)
\right\} $, the proof is completed.
\end{proof}

Let $S=\{S_k^{(\rho)}\}$ be a $\rho$-dyadic martingale. For $n=1,2\dotsc$, consider $N_n=N_n(S)$ defined as
$$
N_n(x')=\left(\sup_{k\le n} ( S_k^{(\rho)}(x)-\frac{1}{2}\langle S
\rangle_k^2(x) ) \right)^+.
$$
Here $x^+=\max \{x,0\}$, $x\in\mathbb{R}$. From Lemma~\ref{lemma2.2} we easily deduce the following result.

\begin{lemma}\label{lemma2.3}
Let $S =  \{S_k^{(\rho)}\}$ be a $\rho$-dyadic martingale. Fix $I_0
\in\mathcal {D}_0(\rho)$ and assume $S_{0}^{(\rho)}\equiv 0$ on
$I_0$. Then for any $0< \alpha < 1$ and any integer $n \geq 1$, one
has
$$
\int_{I_0} \exp (\alpha N_n(x))\,dx\le \frac{1}{1-\alpha}|I_0|.
$$
\end{lemma}

\begin{proofwqed}
Fix the integer $n \geq 1$ and $0< \alpha <1$.
Since
$$
\exp (\alpha N_n(x))\le 1+\exp \left(\alpha\sup_{k\le n} \left(
S_k^{(\rho)}(x)-\frac{1}{2}\langle S \rangle_k^2(x)\right)\right)
$$
we have
$$
\int_{I_0} \exp (\alpha N_n(x))\,dx \le |I_0|+\int_{I_0} \exp
\left(\alpha\sup_{k\le n} \left( S_k^{(\rho)}(x)-\frac{1}{2}\langle
S \rangle_k^2(x)\right)\right)\,dx.
$$
By Lemma~\ref{lemma2.2}, the integral in the right hand side term is
bounded by
$$
\alpha |I_0|\int_0^{\infty} e^{\alpha \lambda}
e^{-\lambda}\,d\lambda=\frac{\alpha}{1-\alpha}|I_0|.\qedif
$$
\end{proofwqed}

Let $S=\{S_k^{(\rho)}\}$ be a $\rho$-dyadic martingale. Fix $I_0
\in\mathcal {D}_0(\rho)$. It is clear that for any $k \geq 1$,
orthogonality gives that
$$
\int_{I_0} (S_k^{(\rho)} (x) - S_0^{(\rho)} (x))^2 dx = \int_{I_0}
\langle S \rangle_{k}^2(x) dx
$$
We end this section with a local version of this result which will
be used later.

\begin{lemma}\label{lemma2.4}
Let $S=\{S_k^{(\rho)}\}$ be a $\rho$-dyadic martingale. Fix $I_0
\in\mathcal {D}_0(\rho)$ and assume $S_{0}^{(\rho)}\equiv 0$ on
$I_0$. Consider the set $E= \{x \in I_0 : \sup_k |S_k^{(\rho)}
(x)|\leq 1 \}$. Then there exists an absolute constant $C$, independent of $S$ and $I_0$,  such that
$$
\int_{E} \langle S \rangle_{\infty}^2(x) \,dx\le C.
$$
\end{lemma}

\begin{proofwqed}

Write $S_k = S_k^{(\rho)}$ and let $S_k (I)$ denote the constant
value of $S_k$ in the interval $I \in \mathcal{D}_k (\rho) $. Let
$\mathcal{G}$ be the family of maximal $\rho$-dyadic intervals $I$
such that $|S_k (I)| > 1$. Here $k$ is the integer for which $I \in
\mathcal{D}_k (\rho) $. It is clear that the set $E$ does not
intersect the interior of any interval in $\mathcal{G}$. Let
$\mathcal{G}_1 $ be the subcollection of $\rho$-dyadic intervals $I
\in \mathcal{G}$ for which $|S_k (I)| > 10$, where again $I \in
\mathcal{D}_k (\rho) $. We claim that if $I \in \mathcal{G}_1 $ and
$I'$ is the $\rho$-dyadic brother of $I$, that is $|I'|=|I|$ and $I
\cup I' \in \mathcal{D} (\rho) $, then $I' \in \mathcal{G}$.
Actually if $I^* = I \cup I'$ is the $\rho$-dyadic father of $I$, by
maximality, $I^* $ is not in $\mathcal{G}$, that is, $|S_l (I^*)|
\leq 1$, where $I^* \in \mathcal{D}_l (\rho) $. Since $S_l (I^*) =
(S_{l+1} (I) + S_{l+1} (I')) / 2$ and $|S_{l+1} (I)|> 10$, we deduce
that $|S_{l+1} (I')|>8$. Hence $I' \in \mathcal{G}$ as claimed. So,
$\rho$-dyadic brothers of intervals in $\mathcal{G}_1 $ are in
$\mathcal{G}$. Hence the interiors of $\rho$-dyadic fathers of
intervals in $\mathcal{G}_1 $ do not intersect $E$. Now, stop the
martingale $S$ either at intervals which are $\rho$-dyadic fathers
of intervals in $\mathcal{G}_1 $ or at intervals in $\mathcal{G}$.
Let $S^{\tau}$ be the corresponding stopped martingale and observe
that $\| S^{\tau} \|_\infty \leq 10$. Since for any $x \in E$ one
has $S_k (x) = S^{\tau}_k (x)$ for any $k$, we deduce that $ \langle
S \rangle_{\infty}(x) = \langle S^{\tau}\rangle_{\infty}(x)$ for any
$x \in E$. Thus
$$
\int_{E} \langle S \rangle_{\infty}^2(x) dx = \int_{E} \langle
S^{\tau}\rangle_{\infty}^2(x) \,dx\leq \int_{I_0} \langle
S^{\tau}\rangle_{\infty}^2(x) \,dx = \int_{I_0} |S^{\tau}_{\infty}
(x)|^2 \,dx \le 100 \rho.
$$
\end{proofwqed}

\section{The Law of the Iterated Logarithm}\label{section3}

Fix $1\le \rho< 4$. Given a function~$g$ defined in the real line,
we denote by~$S(g)$ the $\rho$-dyadic martingale
$S(g)=\{S_k^{(\rho)}(g)\}_k$ defined as
\begin{equation}\label{eq3.1}
S_k^{(\rho)} (g) (x)=\frac{g(b)-g(a)}{b-a}, \quad
k=0,1,2,\dotsc,
\end{equation}
for $x \in I=[a,b)\in \mathcal{D}_k(\rho)$. Let $f$ be a function
defined at almost every point~$x\in\mathbb{R}$. Fixed $s \in
\mathbb{R}$, consider the function~$f_s$ defined by $f_s
(x)=f(x-s)$, $x\in\mathbb{R}$, and the $\rho$-martingale
$S(f_s)=\{S_k^{(\rho)}(f_s)\}_k$ which is well defined a.e.\
$(s,\rho)\in\mathbb{R}\times [1,4]$. Let $\fint_E f(x) dx$ denote
the mean of a locally integrable function $f$ on the measurable set
$E$, that is, $\fint_E f(x) dx = (\int_E f(x) dx)/ |E|$. Next
auxiliary result tells that the mean divided difference
$\tilde{\Delta} (f)$ and the square function $A(f)$ defined in the
Introduction, can be understood, respectively,  as means of the
martingales~$S(f_s)$ and their quadratic variation.

\begin{lemma}\label{lemma3.1}
Let $f \in L^1_{\mathrm{loc}}(\mathbb{R})$. For $s \in \mathbb{R}$
consider the function~$f_s$ defined by $f_s(x)=f(x-s)$,
$x\in\mathbb{R}$, and the $\rho$-dyadic martingale
$\{S_k^{(\rho)}(f_s)\}_k$ as defined in \eqref{eq3.1}. For $0<y<2$
let $N=N(y)$ be the unique integer such that $H=H(y)=2^Ny$ satisfies
$1\le H<2$.
\begin{enumerate}
\item[(a)] For any $x\in\mathbb{R}$ and $0<y<2$, one has
$$
 \tilde{\Delta} (f)(x,y)=\int^y_{y/2}
\fint^{x+h}_{x-h}\frac{f(s+h)-f(s-h)}{2h}\,ds
\frac{dh}{h}=\int^{2H}_H \fint^{\rho}_0 S_N^{(\rho)}(f_s)(x+s)\,ds
\frac{d\rho}{\rho}.
$$
\item[(b)] Assume $f\in L^2_{\mathrm{loc}}(\mathbb{R})$. Then for any $x\in\mathbb{R}$ and any $0<y<1$ one has
$$
\int_y^H\fint_{x-h}^{x+h}\Delta_2^2(f)(s,h)\,ds\frac{dh}{h}=\int^{2H}_H\fint^{\rho}_0\langle
S^{(\rho)}(f_s)\rangle^2_N  (x+s)\,ds\frac{d\rho}{\rho}.
$$
\end{enumerate}
\end{lemma}

\begin{proofwqed}
(a) Fix $1\le \rho<4$ and $0<y<2$. An easy calculation shows
$$
\fint^{\rho}_0 S_N^{(\rho)}(f_s)(x+s)\,ds=\fint_0^{2^{-N}\rho}\Delta
(f)(x+t-2^{-N-1}\rho,2^{-N-1}\rho)\,dt\quad ,  x\in\mathbb{R}.
$$
Integrating this identity with respect $d\rho/\rho$ and introducing the variables $s=x+t-2^{-N-1}\rho$, $h=2^{-N-1}\rho$, we deduce
$$
\int^{2H}_H\fint^{\rho}_0 S_N^{(\rho)}(f_s)(x+s)\,ds\frac{d\rho}
{\rho}=\int^y_{y/2}\fint^{x+h}_{x-h}\Delta (f)(s,h)\,ds\frac{dh}{h}
$$
which proves (a). To  prove (b) fix $1\le\rho<4$ and observe that
for any function~$f$ defined in $[0,\rho]$ and any $k=0,1,2,\dotsc$,
one has
$$
\left(S_{k+1}^{(\rho)}(f)(x)-S_k^{(\rho)}(f)(x)\right)^2=\Delta^2_2(f)\left(\frac{a+b}{2},2^{-k-1}\rho\right),
$$
where $a=a(x)$, $b=b(x)$ are defined by $x\in [a,b)\in\mathcal{D}_k(\rho)$. Using this identity, an easy calculation shows
\begin{equation*}
\begin{split}
&\fint^{\rho}_0\left(S_{k+1}^{(\rho)}(f_s)(x+s)-S_k^{(\rho)}(f_s)(x+s)\right)^2\,ds   = \\*[5pt]
=&\fint_0^{2^{-k}\rho}\Delta_2^2(f)(x+t-2^{-k-1}\rho, 2^{-k-1}\rho)\,dt,\quad\text{a.e.\ }x\in\mathbb{R}.
\end{split}
\end{equation*}
Integrating this identity with respect $d\rho/\rho$ and introducing the variable $h=2^{-k-1}\rho$, we deduce
\begin{equation*}
\begin{split}
&\int_H^{2H} \fint^{\rho}_0
\left(S_{k+1}^{(\rho)}(f_s)(x+s)-S_k^{(\rho)}(f_s)(x+s)\right)^2\,ds\frac{d\rho}{\rho}  = \\*[5pt]
=&\int^{2^{-k}H}_{2^{-k-1}H} \fint^{2h}_0 \Delta_2^2
(f)(x+t-h,h)\,dt\frac{dh}{h},\quad\text{a.e.\ }x\in\mathbb{R}.
\end{split}
\end{equation*}
Adding on $k=0,\dotsc,N-1$, we deduce
$$
\int^{2H}_H\fint^{\rho}_0 \langle S^{(\rho)} (f_s) \rangle^2_N
(x+s)\,ds\frac{d\rho}{\rho}=\int^H_{2^{-N}H}\fint^{2h}_0\Delta_2^2(f)(x+t-h,h)\,dt\frac{dh}{h},\quad
\,  x\in\mathbb{R}.\qedif
$$
\end{proofwqed}

Denote by ${\tilde A}^2 (f) (x,y)$ the left term in the identity in part (b) of Lemma~\ref{lemma3.1}, that is
$$
{\tilde A}^2 (f) (x,y) = \frac{1}{2}  A^2 (f) (x,y) +
\int_1^H\fint_{x-h}^{x+h}\Delta_2^2(f)(s,h)\,ds\frac{dh}{h} \quad ,
\,  x \in \mathbb{R},  0<y<1
$$
Note that there exists an absolute constant $C>0$ such that
\begin{equation}\label{eq3.2}
|{\tilde A}^2 (f) (x,y) -  \frac{1}{2} A^2 (f) (x,y)| < C
\int_{x-2}^{x+2} |f(t)|^2 dt
\end{equation}
For $f\in L^2_{\mathrm{loc}}(\mathbb{R})$ and $0<h<1$, consider
$$
N(f)(x,h)=\frac{1}{\ln 4} \sup_{1 \ge y\ge h}
\left(\tilde{\Delta} (f)(x,y)-\tilde
{\Delta} (f)(x,H(y))-\frac{1}{2} {\tilde A}^2(f)(x,y)\right),
$$
where $H(y)$ is defined in  Lemma~\ref{lemma3.1}. Recall that $1
\leq H(y) < 2$. A version of Lemma~\ref{lemma2.3} in the continuous
setting is given in the following result.

\begin{lemma}\label{lemma3.2}
Let $f\in L^2_{\mathrm{loc}}(\mathbb{R})$. For any $0<\alpha<1$, $0<h<1$ and any interval~$I\subset\mathbb{R}$ with $|I|=1$ one has
$$
\int_I\exp (\alpha N(f)(x,h))\,dx\le\frac{C}{1-\alpha}
$$
where $C>0$ is a universal constant independent of $\alpha$,  $h$, $I$ and $f$.
\end{lemma}

\begin{proof}
Fix $0<h<1$. For $h\le y< 1$, let $N(y)$ be the integer defined in
the statement of Lemma~\ref{lemma3.1}, that is, $N(y)$ is the unique
integer satisfying $H(y)=2^{N(y)}y \in [1,2)$. Lemma~\ref{lemma3.1}
gives that for any $x\in\mathbb{R}$ one has
\begin{multline*}
\tilde {\Delta} (f)(x,y)-\tilde{\Delta} (f)(x, H(y))-\frac{1}{2}
 {\tilde A}^2(f)(x,y)\\*[5pt]
=\int^{2H(y)}_{H(y)}\fint^{\rho}_0\left(S^{(\rho)}_{N(y)}(f_s)(x+s)-S_0^{(\rho)}(f_s)(x+s)-\frac{1}{2}\langle
S^{(\rho)}(f_s)\rangle^2_{N(y)}(x+s)\right)\,ds \frac{d\rho}{\rho}.
\end{multline*}
Since $y\ge h$ we have $N(y)\le N(h)$ and we deduce
$$
N(f)(x,h)\le \frac{1}{\ln 4} \int^4_1\fint_0^{\rho}
N^{(\rho)}_{N(h)}(x+s)\,ds\frac{d\rho}{\rho},
$$
where
$$
N_n^{(\rho)}(x)= N_n^{(\rho)} (f_s) (x) =\sup_{k\le
n}\left(S_k^{(\rho)}(f_s)(x)-S_0^{(\rho)}(f_s)(x)-\frac{1}{2}\langle
S^{(\rho)}(f_s)\rangle^2_k(x)\right)^+.
$$
Fix $0<\alpha<1$. Jensen's inequality and Fubini Theorem give
$$
\int_I\exp (\alpha N(f)(x,h))\,dx \le \frac{1}{\ln 4}
\int^4_1\fint^{\rho}_0\int_I\exp (\alpha
N^{(\rho)}_{N(h)}(x+s))\,dx\,ds \frac{d\rho}{\rho}.
$$
Lemma~\ref{lemma2.3} gives that there exists a universal constant
$C>0$ such that for every $s$ and $\rho$, one has
$$
\int_I\exp (\alpha N^{(\rho)}_{N(h)}(x+s))\,dx\le \frac{C}{1-\alpha}
$$
and the proof is completed.
\end{proof}

The main technical step in the proof of our results is the following good $\lambda$-inequality with subgaussian decay.

\begin{lemma}\label{lemma3.3}
Let $f\in L^2_{\mathrm{loc}}(\mathbb{R})$. For any $N,M>0$ with $M^2 > 4N$ and any interval $I$ of unit length consider the set $E=E(M,N)$ of points $x\in I$ for which  there exists $h=h(x)>0$ with $0<h<1$ such that
$$
\sup_{1 \ge y \ge h} (\tilde{\Delta} (f)(x,y)-\tilde {\Delta} (f)(x, H(y)))\ge
M,
$$
and
$$
 {\tilde A}^2(f)(x,h)\le N.
$$
Then
$$
|E|\le C\frac{M^2}{2N}\exp \left( \frac{-M^2}{2N \ln 4} \right).
$$
Here $C$ is an absolute constant independent of $N$, $M$, $I$ and $f$.
\end{lemma}

\begin{proof}
One can assume that there exists $h_0 >0$ such that $h(x)\ge h_0 >0$
for any $x\in E$. Observe that for any $\lambda>0$ and any $x\in E$
one has $(\ln 4 ) N(\lambda f)(x,h_0 )\ge \lambda M-\lambda^2N/2$.
Lemma~\ref{lemma3.2} gives that for any $0<\alpha<1$ one has
$$
|E|\exp (\alpha(\lambda M-\lambda^2N/2) / \ln 4)\le
\frac{C}{1-\alpha}.
$$
Taking $\lambda =M/N$ one gets
$$
|E|\le \frac{C\exp (\frac{-\alpha M^2}{2N \ln 4} )}{1-\alpha}.
$$
The optimal choice $\alpha=1-2N \ln 4 /M^2$ finishes the proof.
\end{proof}

Using the subgaussian estimate of Lemma~\ref{lemma3.3}, an standard Borel-Cantelli argument gives
the Law of the Iterated Logarithm stated in the Introduction as Theorem 1.


\begin{proof}[Proof of Theorem 1]
By \eqref{eq3.2}, in the statement $A^2 (f) (x,h)$ can be replaced
by $2 {\tilde A}^2 (f) (x,h)$. Fix $R>1$, $L>1$ and $k=1,2,\dotsc$.
Consider the set~$E_k$ of points $x\in [-L, L]$ for which there
exists $h=h(x)\in (0,1)$ with $R^k\le {\tilde A}^2(f)(x,h)<R^{k+1}$
and
$$
\tilde{\Delta} (f)(x,h)>R 2 \sqrt{ (\ln 2 ) {\tilde
A}^2(f)(x,h)\ln\ln {\tilde A}^2(f)(x,h)}.
$$
Since there exists an absolute constant $C>0$ such that
$$
|\tilde{\Delta} (f)(x,H(h))|\le C \int^{x+2}_{x-2}|f(t)|\,dt,
$$
Lemma~\ref{lemma3.3} applied with $N=R^{k+1}$ and $M=R\sqrt{4 (\ln
2) R^k\ln\ln R^k}$ gives that for $k$ sufficiently large one has
$$
|E_k|\le C(R)(\ln k)k^{-R} |L|
$$
where $C(R)$ denotes a constant depending on $R$. Thus $\sum |E_k|<\infty$ and we deduce
$$
\left|\bigcap_m\bigcup_{k>m} E_k\right|=0.
$$
So, almost every point $x\in [-L, L]$ is at most, in a finite number
of sets~$E_k$. In particular for almost every $x\in \{x\in [-L, L]:
A(f)(x)=\infty\}$ one has
$$
\tilde{\Delta} (f)(x,h)<R 2 \sqrt{ (\ln 2) {\tilde A}^2(f)(x,h)\ln
\ln {\tilde A}^2(f)(x,h)}
$$
if $h>0$ is sufficiently small. Since $L$ can be taken arbitrarily
large, one deduces that
$$
\limsup_{h\to 0^+}\frac{\tilde{\Delta}
(f)(x,h)}{\sqrt{{\tilde A}^2(f)(x,h)\ln\ln {\tilde A}^2(f)(x,h)}}\le 2R \sqrt{\ln 2}
$$
at almost every $x\in\{x\in \mathbb{R}: A(f)(x)=\infty\}$. Since the previous estimate also holds for $-f$ and any
$R>1$, the proof is completed.
\end{proof}

For future reference it is useful to state the following version of
Lemma~\ref{lemma3.3}.

\begin{lemma}\label{lemma3.4}
Let $f\in L^2_{\mathrm{loc}}(\mathbb{R})$ and let $I \subset
\mathbb{R}$ be an interval. For any $N,M>0$ with $M^2 > 4N$, consider the set
$E=E(M,N)$ of points $x\in I$ for which there exists $h=h(x)$ with  $0< h < |I|/2$ such that
$$
\sup_{|I| > y \ge h} (\tilde{\Delta} (f)(x,y)-\tilde {\Delta}
(f)(x, H(y) |I|))\ge M,
$$
and
$$
{\tilde A}_{|I|}^2(f)(x,h) = \int_h^{H(h)|I|} \int_{x-t}^{x+t} \Delta_2^2 (f)
(s, t ) \frac{ds dt}{t^2} \le N.
$$
Then
$$
|E|\le C\frac{M^2}{2N}\exp \left( \frac{-M^2}{2N \ln 4} \right) |I|.
$$
Here $C$ is a universal constant.
\end{lemma}

\section{Sobolev Spaces}\label{sec4}

In this Section we will show that Sobolev spaces can be described in
terms of size conditions on the square function as stated in Theorem
3 of the Introduction. For $1<p<\infty$ let $W^{1,p} (\mathbb{R})$
be the Sobolev space of functions~$f\in L^p (\mathbb{R})$ for which
the distributional derivative~$f'$ is a function  in $L^p
(\mathbb{R})$. Equivalently, a function~$f\in L^p (\mathbb{R})$ is
in $W^{1,p} (\mathbb{R})$ if and only if
$$
\sup_{|h|\le 1}\left|\frac{f(x+h)-f(x-h)}{2h}\right|\in L^p
(\mathbb{R}).
$$
The necessity is clear because $\left|f(x+h)-f(x-h)\right|/ 2 |h|$
is bounded by the Hardy-Littlewood maximal function of~$f'$. The
sufficiency can be proved as follows. There exists $h_n\to 0$ such
that $h_n^{-1}(f(x+h_n)-f(x-h_n))$ converges weakly in $L^p
(\mathbb{R})$ to a certain function~$g\in L^p (\mathbb{R})$. Then
one may easily check that $g$ is the distributional derivative
of~$f$. Hence $f\in W^{1,p} (\mathbb{R})$.

We now prove Theorem 3 stated in the Introduction.


\begin{proof}[Proof of Theorem 3]
Let $f\in W^{1,p} (\mathbb{R})$. In the case $2\le p<\infty$, a
simple argument based in Lemma~\ref{lemma3.1} will give that $A(f)
\in L^p (\mathbb{R})$. Let $M(f')$ be the Hardy-Littlewood maximal
function of $f'\in L^p (\mathbb{R})$. Since for any $\rho\in [1,4]$
and any $s\in [0,\rho]$ we have
$$
M(S_k^{(\rho)} (f_s))(x+s)\le M (f')(x),\quad x\in\mathbb{R},
$$
we deduce that $M(S_k^{(\rho)}(f_s))\in L^p (\mathbb{R})$ and $\|M
(S_k^{(\rho)}(f_s))\|_p\le C_1(p) \|f' \|_p$. Hence $\|\langle
S^{(\rho)}(f_s)\rangle^2_{\infty}\|_{p/2}\le C_2(p) \|f ' \|^2_p$.  By
Lemma~\ref{lemma3.1}, a.e.\ $x\in\mathbb{R}$ one has
$$
\frac{1}{2} A^2(f)(x)\leq \int_1^{4}  \fint^{\rho}_0\langle
S^{(\rho)}(f_s)\rangle^2_{\infty} (x+s)\,ds \frac{d\rho}{\rho}.
$$
Now, if $p\ge 2$, Minkowski inequality gives
$\|A(f)\|_p^2=\|A^2(f)\|_{p/2}\le C_3 (p) \|f' \|^2_p$ and finishes
the proof.
In the case $1 < p <  2$, we will adapt an argument of Fefferman and
Stein (\cite[p.~162]{FS}). Let $f\in W^{1,p} (\mathbb{R})$ and
take $\lambda >0$. Consider the closed set $E=\{x\in \mathbb{R}:
M(f')(x)\le\lambda\}$. The main estimate of the proof is the
following good-$\lambda$ inequality
\begin{equation}\label{eq4.1}
|\{x\in E: A(f)(x)>\lambda\}|\le C |\mathbb{R} \setminus E|+
\frac{C}{\lambda^2} \int^{\lambda}_0 t|\{x\in \mathbb{R}:
M(f')(x)>t\}|\,dt ,
\end{equation}
where $C$ is a universal constant, independent of~$f$ and $\lambda$.
To prove \eqref{eq4.1} we will show that there exists an absolute constant $C_1 >0$ such that for any  $1\le\rho\le 4$ and any $0 \leq s \leq \rho $, one has
\begin{equation}\label{eq4.2}
\int_{E } \langle S^{(\rho)} (f_s)\rangle_{\infty}^2
(x+s)\,dx\le C_1  \lambda^2|\mathbb{R}   \setminus E|+ C_1 \int^{\lambda}_0
t|\{x\in \mathbb{R}  : M(f')(x)\ge t\}|\,dt .
\end{equation}
Once  \eqref{eq4.2} is proved, integrating on $\rho \in [1,4]$ and $s \in [0, \rho]$, Lemma 3.1 gives that
$$
\frac{1}{2} \int_{E } A^2 (f)(x)\,dx\le 4C_1 \lambda^2|\mathbb{R}
\setminus E|+ 4C_1 \int^{\lambda}_0 t|\{x\in  \mathbb{R}:
M(f')(x)\ge t\}|\,dt
$$
and \eqref{eq4.1} would follow taking $C= 8 C_1$. To prove
\eqref{eq4.2} fix $1\le\rho\le 4$ and $0 \leq s \leq \rho $.
Consider the family $G(\rho, s )$ of intervals of the form $[j
2^{-k} \rho - s,  (j+1)2^{-k} \rho - s)$ where $k \geq 0 $ and $j$
are integers. In other words, intervals in $G(\rho, s)$ are
translation of the $\rho$-dyadic intervals by $s$ units. Fix $I_0
\in G(\rho, s)$ of length $\rho$, that is, of the form $I_0 = [ j
\rho - s ,  (j+1) \rho - s)$ for some integer $j$.
We may assume that $ |I_0 \setminus E \cap I_0| < 1/2$. Consider the family
~$\mathcal{A}(\rho)=\mathcal{A}(\rho,\lambda,f)$ of maximal
 intervals in the family $G(\rho, s)$ contained in $I_0 \setminus E$. Then
\begin{equation}\label{eq4.3}
\sum_{I\in \mathcal{A}(\rho)}|I|=|I_0 \setminus E|.
\end{equation}
 Consider the martingale
$\{S_k^{(\rho)}(f_s) (x+s) \}_k$ and stop it at the intervals of the
family~$\mathcal{A}(\rho)$. Let $S^{(\rho),\tau}$ be the
corresponding stopped martingale. Orthogonality gives
\begin{equation}\label{eq4.4}
\int_{I_0} \langle S^{(\rho),\tau}\rangle^2_{\infty}
(x+s)\,dx=\int_{I_0} \left( S_{\infty}^{(\rho),\tau}
(x+s)-S_0^{(\rho),\tau}(x+s)\right)^2\,dx=A+B,
\end{equation}
where
\begin{align*}
A&=\sum_{I\in \mathcal{A}(\rho)}\int_I \left( S_{k(I)}^{(\rho)}
(f_s)(x+s)-S_0^{(\rho)} (f_s)(x+s)\right)^2\,dx,\\*[5pt] B&= \int_{E \cap I_0}
\left( S_{\infty}^{(\rho)}(f_s)(x+s)-
S_0^{(\rho)}(f_s)(x+s)\right)^2\,dx.
\end{align*}
Here $k(I)$ is the integer satisfying $2^{-k(I)} \rho =|I|$. Fix
$I\in \mathcal{A}(\rho)$. By maximality, its $(\rho , s)$-dyadic father
$\tilde{I} \in G(\rho , s)$ contains a point $\tilde{x}\in E$. Hence $\int_J |f'|\le
\lambda|J|$ for any interval~$J$ containing $\tilde{x}$. Then
$\int_J|f'|\le 3\lambda|J|$ for any interval~$J$ with $J\cap
I\ne\emptyset$ and $|J| \geq |I|$. We deduce that
$|S_{k(I)}^{(\rho)}(f_s)(x+s)| + |S_{0}^{(\rho)}(f_s)(x+s)|\le
6\lambda$ for any $x\in I$. Therefore \eqref{eq4.3} gives
$$
A\le 36 \lambda^2|I_0 \setminus E|.
$$
Consider $F(\rho,s)(t)=\{x\in E \cap I_0 :
|S_{\infty}^{(\rho)}(f_s)(x+s)- S_0^{(\rho)}(f_s)(x+s)|>t\}$. Since
$|S_{\infty}^{(\rho)}(f_s)(x+s)|+ |S_0^{(\rho)}(f_s)(x+s)|\le 2 M(f') (x) \le
2 \lambda$ for $x\in E$, we have $F(\rho,s)(t) \subset \{ x \in E \cap I_0 : M(f') (x) > t / 2  \}$ and
$$
B\le 2 \int^{2 \lambda}_0 t|F(\rho,s)(t)|\,dt \le 2 \int_0^{2
\lambda} t|\{x\in E \cap I_0 : M(f')(x)\ge t/2 \} | dt
$$
%
Since $S^{(\rho),\tau}\equiv S^{(\rho)}$ on~$E \cap I_0$, identity
\eqref{eq4.4} gives
$$
\int_{E \cap I_0} \langle S^{(\rho)} (f_s)\rangle_{\infty}^2
(x+s)\,dx\le 36 \lambda^2|I_0  \setminus E|+ 8 \int^{\lambda}_0
t|\{x\in I_0 : M(f')(x)\ge t\}|\,dt .
$$
Adding this estimate over all $I_0 \in G(\rho, s)$ of length $\rho$, estimate \eqref{eq4.2} follows. Thus \eqref{eq4.1} is proved.
The rest of the proof is easy. From \eqref{eq4.1} it follows that
\begin{equation*}
\begin{split}
 \int_{\mathbb{R}} A^p(f)(x)\,dx&=p \int^{\infty}_0
 \lambda^{p-1}|\{x\in \mathbb{R}:
 A(f)(x)>\lambda\}|\,d\lambda \leq \\*[5pt] \le Cp\int_0^{\infty}
 \lambda^{p-1} |\{x\in \mathbb{R}: M(f')(x)\ge \lambda\}|\,d\lambda
 &+ Cp\int^{\infty}_0 \lambda^{p-3} \int_0^{\lambda}t|\{x\in
 \mathbb{R} :M(f')(x)\ge t\}|\,dt\,d\lambda .
 \end{split}
 \end{equation*}
Since $p<2$ each term is bounded by $C(p) \|M(f')\|^p_p$ and hence
$\|A(f)\|_p \leq C_1 (p) \|f' \|_p$.

In the case $1< p \leq 2$ the converse follows easily from
Lemma~\ref{lemma3.1}. Actually Holder's
inequality gives that
$$
\int_{\mathbb{R}} \int_1^{2} \int_0^{\rho} \langle
S^{(\rho)}(f_s)\rangle^p_{\infty} (x + s) \frac{ds d \rho}{{\rho}^2}
dx \le C(p) \int_{\mathbb{R}} ( \int_1^{2} \int_0^{\rho} \langle
S^{(\rho)}(f_s)\rangle^2_{\infty} (x+ s ) \frac{ds d
\rho}{{\rho}^2})^{p/2} dx
$$
Now part (b) of  Lemma~\ref{lemma3.1}, applied with $y = 2^{-N}$ and
letting $N \to \infty$, gives that
$$
\int_{\mathbb{R}} \int_1^{2} \int_0^{\rho} \langle
S^{(\rho)}(f_s)\rangle^p_{\infty} (x + s) \frac{ds d \rho}{{\rho}^2}
dx \le C(p) \|A(f)  \|_p^p
$$
Fubini's Theorem gives that  almost every $\rho \in [1, 2]$, $ s \in [0, \rho]$, the
function $h_{s, \rho}$ defined as $h_{s, \rho} (x) =  \langle
S^{(\rho)}(f_s)\rangle_\infty (x+s) $ is in $L^p (\mathbb{R})$ and
one can choose $\rho$ and $s$ such that $\| h_{s, \rho} \|_p  \leq  C(p)
\|A(f)  \|_p$. Then the maximal function $M(x)= M_{\rho , s } (x) =
\sup_k |S_k^{(\rho)}(f_s) (x+s)|$ is in $L^p (\mathbb{R})$ and the
limit function $h$ defined by $h(x) = \lim_{k \to \infty}
S_k^{(\rho)}(f_s) (x+s) $ is in $L^p (\mathbb{R})$. It is easy to
see that $h$ is the distributional derivative of $f$ and hence $f
\in W^{1,p}  (\mathbb{R})$. Moreover $\|f'\|_p \le \|M \|_p \le C(p) \|A(f) \|_p$.

Let us now consider the case $2 \leq p < \infty$. We first show that
there exists a constant $C(p) >0 $ such that for any $f \in W^{1,
p}  (\mathbb{R})$ one has
\begin{equation}\label{eq4.5}
\|f' \|_p \le C(p) \| A(f) \|_p
\end{equation}
Let $f \in W^{1, p} (\mathbb{R})$ and let $M(f')$ be the Hardy-Littlewood
maximal function of $f'$. Fixed $\lambda >0$ consider the open set
${\cal U} = \{x \in \mathbb{R} : M(f')(x) > \lambda \}$. Write
${\cal U} = \cup I_j$ where $\{I_j \}$ is a collection of pairwise
disjoint open intervals. Since the end points of any $I_j$ are not in
$ {\cal U}$, we have that
$$
\int_J |f'(x)| dx \leq \lambda |J|
$$
for any interval $J$ containing an end point of any $I_j$. Thus for
any point $x \in I_j$,  any $h \geq |I_j|  $ and any $s \in (x-h,
x+h)$ one has $|f(s+h) - f(s-h)| < 2 \lambda h$. Hence $|
\tilde{\Delta} (f)(x,H |I_j|) | \leq (\ln 2) \lambda$ for any $x \in
I_j$ and any $1 \leq H \leq 2$. Fix $j$ and apply the subgaussian
estimate of Lemma~\ref{lemma3.4} to deduce that there exists a
universal  constant $C>0$ such that for any $0< \varepsilon < 1$ one
has
$$
| \{x \in I_j : |f'(x)| > 10 \lambda , A(f) (x) \le \varepsilon
\lambda  \} | \le C \exp (-1/C {\varepsilon^2}) |I_j|
$$
Note that if $|f' (x)| > 10 \lambda $ then $x \in {\cal U}$. So,
adding the previous estimate over $j$ one gets
\begin{equation}\label{eq4.6}
| \{x \in \mathbb{R}  : |f'(x)| > 10 \lambda , A(f) (x) \le
\varepsilon \lambda \}  | \le C \exp{(-1/C {\varepsilon}^2)} |{\cal
U}|
\end{equation}
The rest of the proof of estimate ~\eqref{eq4.5} is standard. Write
$ \|f' \|_p^p  \leq A + B $ where
\begin{align*}
A&= p \int_0^\infty \lambda^{p-1} |\{x \in \mathbb{R} : |f'(x)| > 10
\lambda , A(f)(x) \leq \varepsilon \lambda  \}| d \lambda \quad ,
\\*[5pt] B&= p  \int_0^\infty \lambda^{p-1} |\{x \in  \mathbb{R} : A(f) (x) > \varepsilon \lambda \}| d \lambda .
\end{align*}
Estimate~\eqref{eq4.6} and the boundedness of the Hardy-Littlewood
maximal function in $L^p$ give that there exists a constant $C_1
(p)$ only depending on $p$ such that $ A \le C_1 (p) \exp{(-1/C
{\varepsilon}^2 )} \|f' \|_p^p $. It is clear that $B \leq
C(p,\varepsilon) \|A(f) \|_p^p$. We deduce that
$$
 \|f' \|_p^p \le C_1 (p) \exp{(-1/C {\varepsilon}^2)} \|f' \|_p^p + C(p,\varepsilon) \|A(f)  \|_p^p
$$
Choosing $\varepsilon >0$ small enough so that $C_1 (p) \exp{(-1/C
{\varepsilon}^2)} < 1$, estimate~\eqref{eq4.5} follows. The rest of
the proof is now easy. Let $f \in L^p (\mathbb{R})$ with $A(f) \in
L^p (\mathbb{R})$. Let $\varphi$ a smooth positive even function
with $\|\varphi\|_1 = 1$. For $0< \varepsilon < 1$ consider $\varphi_\varepsilon (x) =
\varepsilon^{-1} \varphi (x/\varepsilon)$ and $f_\varepsilon = f *
\varphi_\varepsilon$. Schwarz's inequality gives
$$
A^2 (f_\varepsilon)(x) \leq (A^2 (f) * \varphi_\varepsilon ) (x),  \quad x
\in \mathbb{R}  \,  .
$$
If $p \geq 2$, Holder's inequality gives $ \|A(f_\varepsilon)\|_p
\leq \| A(f) \|_p $. Now estimate~\eqref{eq4.5} gives that $\|( f *
\varphi_\varepsilon)'\|_p \leq C(p)\| A(f) \|_p $ for any $0 <
\varepsilon < 1$. We deduce that there exists a subsequence
$\varepsilon_n \to 0$ such that $(f * \varphi_{\varepsilon_n})'$
converges weakly in $L^p (\mathbb{R})$ to a function $h \in L^p
(\mathbb{R})$. It is easy to show that $h$ is the distributional
derivative of $f$. Hence $f \in W^{1,p}  (\mathbb{R}) $. Moreover $\|f' \|_p  = \|h
\|_p \leq C(p) \|A(f) \|_p$.
\end{proof}



\section{Pointwise differentiability}\label{sec5}

This Section is devoted to the proof of Theorem 2.


\begin{proof}[Proof of Theorem 2]
We first show that almost every point in $A$ is in $B$. If $f$ is
differentiable at~$x$, for $\varepsilon = \varepsilon (x)
>0 $ sufficiently small one has
 $\sup \{ |\Delta_2 f(x,h)| : 0< h <  \varepsilon  \} <\infty $. Using the notation of equation ~\eqref{eq3.1}, consider the $\rho$-dyadic
 martingale $\{S_k^{(\rho)}(f_s)\}_k$  of the divided differences of the function~$f_s$ defined as $f_s(t)=f(t-s)$, $t\in\mathbb{R}$. It is clear that $A\subseteq \bigcup\limits^{\infty}_{N=1}A_N$ where
$$
A_N=\left\{x\in {\cal U}: \sup_{s,\rho,k}
|S_k^{(\rho)}(f_s)(x+s)|\le N\right\}.
$$
Fix $N\ge 1$ and let $E\subset A_N$ be a bounded measurable set.
Lemma~\ref{lemma2.4} gives that for any $1\leq \rho \leq 4$ and any
$0\leq s\leq \rho$, one has
$$
\int_E \langle S^{(\rho)}(f_s)\rangle_{\infty}^2(x+s)\,dx\le CN^2 |E|.
$$
Integrating in $\rho\in [1,4]$ and $s\in [0,\rho]$, Lemma~\ref{lemma3.1} yields
$$
\int_E A^2(f)(x)\,dx\le C_1 N^2 |E|
$$
and hence $A(f)(x)<\infty$ a.e.\ $x\in E$. Hence almost every point of $A_N$ is in $B$.

Let us now show the opposite inclusion, that is, almost every point
in $B$ is in $A$. Fix $N\ge 1$. It is sufficient to show that if $E$
is a bounded measurable set contained in
$$
\left\{x\in {\cal U} : A(f)(x)\le N,\, \sup_{0<h<1/N} |\Delta_2 f(x,h)| \le N \right\} \, ,
$$
then almost every point of $E$ is in $A$. For any $0<\delta <1$,
there exists a subset $E(\delta)\subset E$ with $|E(\delta)|>
(1-\delta)|E|$ and a constant $h_0=h_0(\delta)>0$ such that for any
$x\in E(\delta)$ and any $0<h<h_0$ we have $|(x-h,x+h)\cap E|\ge h$.
We want to show that for any $0< \delta < 1$ almost every point of
$E(\delta)$ is in $A$. Fix $\delta>0$.  Denote by
$\mathbf{1}_{\Gamma(x)}$ the characteristic function of the cone
$\Gamma (x) = \{(s,t) \in {\mathbb{R}}^2 _+ : |s-x|< t <1 \}$.  We
have
\begin{equation*}
\begin{split}
|E|N^2 &\ge \int_E A^2(f)(x)\,dx =\int_E\int_{\mathbb{R}^2_+}
\Delta_2^2 (f)(s,t)\mathbf{1}_{\Gamma(x)}(s,t)\frac{ds dt}{t^2} \,
dx \\*[5pt] &\ge \int_{E(\delta) \times (0,h_0)}
\Delta_2^2(f)(s,t)\left(\int_E\mathbf{1}_{\Gamma(x)}(s,t)\,dx\right)\frac{dt\,ds}{t^2}.
\end{split}
\end{equation*}
Since for any $s \in E(\delta)$ and any $0<t<h_0$, the inner integral is bounded below
by $t$, we deduce
$$
|E|N^2\ge\int_{E(\delta)}\int_0^{h_0} \Delta_2^2(f)(s,t)\frac{dt\,ds}{t}.
$$
In particular at almost every $s \in E(\delta)$ we have
$$
\int^{h_0}_0\Delta^2_2 (f)(s,t)\frac{dt}{t}<\infty.
$$
The classical result by Stein and Zygmund gives that $E(\delta)\subset A$ a.e. This finishes the proof.
\end{proof}


\section{Several variables}\label{sec6}

Given $\xi \in \mathbb{R}^d$ with $|\xi|= 1$, let $\Pi (\xi) =\{x
\in  \mathbb{R}^d : \langle x, \xi \rangle =0\} $ be the hyperplane
in $\mathbb{R}^d$ orthogonal to $\xi$ passing through the origin.
For $x \in \mathbb{R}^d$ denote by $\tilde {x}$ its orthogonal
projection onto $\Pi (\xi)$, that is, $x =  \tilde {x} +  \tilde {s}
\xi$ where $\tilde {x}  \in \Pi (\xi)$ and $\tilde {s} \in
\mathbb{R}$ . Let $f$ be  a function defined in an open set ${\cal
U} \subset \mathbb{R}^d$. For $\tilde {x} \in \Pi (\xi)$ consider
the one variable function $f_{\tilde {x}}$ defined as $f_{\tilde
{x}} (s) = f( \tilde {x} + s \xi)$ for $s \in \{s \in \mathbb{R} :
\tilde {x} + s \xi \in {\cal U} \}$. Assume that $f$ is locally
integrable and consider the mean divided difference in the direction
of $\xi$, denoted by $\tilde {\Delta}_\xi (f)$, defined as $\tilde
{\Delta}_\xi (f)(x,h) = \tilde {\Delta} (f_{\tilde {x}})(\tilde {s},
h )$, where $x = \tilde {x} + \tilde {s} \xi$. In other words, for
$x \in  {\cal U} $ and $0< h < dist (x, \mathbb{R}^d \setminus {\cal
U}) / 2 $,
$$
\tilde {\Delta}_\xi (f)(x,h) =\int^h_{h/2} \int^{t}_{-t}  \frac{f( x
+ s \xi+t \xi)-f(x + s \xi -t \xi)}{2t}\, \frac{ds dt}{2 t^2}  .
$$
It is clear that if the ordinary directional derivative $D_{\xi} (f)
(x)$ at the point $x \in  \mathbb{R}^d$ exists, then $\tilde
{\Delta}_\xi (f)(x,h) $ tends to $c D_{\xi}  (f) (x)$ as $h$ tends
to $0$. Here $c = \ln 2$. Similarly, if $f \in
L^2_{\mathrm{loc}}({\cal U})$, its square function in the direction
$\xi$ is denoted by $A_{\xi} (f)$ and defined by $A_{\xi} (f) (x, h)
= A (f_{\tilde {x}}) (\tilde {s} , h)$, where $x = \tilde {x} +
\tilde {s} \xi$. In other words, for $x \in  {\cal U} $ and $0< h <
h_0 = \min \{1, dist (x, \mathbb{R}^d \setminus {\cal U}) / 2 \}$,
$$
A_{\xi}^2(f)(x,h)=\int_h^{h_0} \int_{-t}^{t}  \left| \frac{f(x + s
\xi + t \xi) + f (x + s \xi - t \xi ) - 2 f(x + s \xi)}{t} \right|^2
\, \frac{ds dt}{t^2}
$$
As before we denote $A_{\xi}^2 (f) (x)
= A_{\xi}^2 (f) (x,0)$. We now prove Theorem 4.

%



\begin{proof}[Proof of Theorem 4]
As before write $x = \tilde {x} + \tilde {s} \xi$ where $ \tilde
{x}  \in \Pi (\xi)$  and $\tilde {s} \in \mathbb{R}$. Consider the one variable function $f_{\tilde x }$ which is defined in an open set ${\cal \tilde {U}} =  {\cal \tilde {U}} (\tilde {x}) \subset   \mathbb{R}$.  For any $ \tilde {x}  \in \Pi (\xi)$, the one dimensional result gives that the sets $\{s \in  {\cal \tilde {U}}  :  f_{\tilde {x}} \text{ is differentiable at } s  \} $ and
 $$
\left\{ s \in  {\cal \tilde {U}} : A (f_{\tilde {x}})(s)<\infty
\text{ and } \sup_{0<t< \varepsilon} \left|\frac{f_{\tilde {x}} (s +
t ) + f_{\tilde {x}} (s - t ) - 2 f_{\tilde {x}} (s)}{t}
\right|<\infty \text{ for some }  \varepsilon = \varepsilon (\tilde
{x} , s) >0 \right\}
$$
can differ at most by a set of length zero. Hence by Fubini's
Theorem part (a) follows. Similarly, the one variable result gives
that for any $\tilde {x}  \in \Pi (\xi)$ one has
$$
\limsup_{h\to 0} \frac{|\tilde{\Delta} (f_{\tilde {x}}
)(s,h)|}{\sqrt{A^2(f_{\tilde {x}} )(s, h)\ln\ln A^2(f_{\tilde {x}}
)(s, h)}} \le  \sqrt{2\ln 2}.
$$
almost every $s \in \{s \in  {\cal \tilde {U}} : A (f_{\tilde {x}} )(s) = \infty \}$. Part (b) follows again by
Fubini's Theorem. Let us now prove part (c). As before for any $\tilde {x}  \in \Pi (\xi)$ consider the function
$f_{\tilde {x}}$. Let $m_{d-1} $ denote Lebesgue measure in $\Pi
(\xi)$. Since $f  \in L^p (\mathbb{R}^d)$ we have $f_{\tilde {x}}
\in L^p (\mathbb{R}) $ almost every ($m_{d-1}$) $ \tilde {x}  \in
\Pi (\xi)$. Assume $D_{\xi}  f  \in L^p (\mathbb{R}^d)$. Fubini's
Theorem gives that for almost every ($m_{d-1}$) $ \tilde {x}  \in
\Pi (\xi)$, the function $f_{\tilde {x}}$ is absolutely continuous
and $f'_{\tilde {x}} \in L^p (\mathbb{R})$. Theorem 3 gives a
constant $C>0$ such that for
 almost every ($m_{d-1}$) $ \tilde {x}  \in \Pi (\xi)$, one has
$$
C^{-1} \| A(f_{\tilde {x}}) \|_{L^p (\mathbb{R})}  \le \|
f'_{\tilde {x}} \|_{L^p (\mathbb{R})}  \le C \|A (f_{\tilde {x}})
\|_{L^p (\mathbb{R})}
$$
Integrating over $\tilde {x}  \in \Pi (\xi)$ we deduce
$C^{-1}\|A_{\xi} (f)\|_p\le \|D_\xi (f) \|_p\le C\|A_{\xi} (f)\|_p$.
Conversely, assume $A_\xi (f)  \in  L^p (\mathbb{R}^d)$. Fubini's
Theorem gives that for almost every ($m_{d-1}$)  point $ \tilde {x}
\in \Pi (\xi)$, $ A (f_{\tilde {x}}) \in L^p (\mathbb{R})$. Theorem
3 gives that $f_{\tilde {x}} \in W^{1,p} (\mathbb{R}) $.  Hence $f$
is absolutely continuous along almost ($m_{d-1}$) every line
parallel to $\xi$ and its directional derivative in the sense of
distributions is $f'_{\tilde {x}}$. Moreover there exists a constant
$C>0$ such that for almost every ($m_{d-1}$) $ \tilde {x} \in \Pi
(\xi)$, one has
$$
C^{-1}\|A(f_{\tilde {x}})\|_{L^p (\mathbb{R})}  \le \| f'_{\tilde {x}} \|_{L^p (\mathbb{R})}  \le C\|A (f_{\tilde {x}})\|_{L^p
(\mathbb{R})}
$$
Integrating over $\tilde {x}  \in \Pi (\xi)$ we deduce $D_{\xi}  f
\in L^p (\mathbb{R}^d)$.
\end{proof}




Let $f \in L^2_ {\mathrm{loc}}(\mathbb{R}^d )$.  Let $S^{d-1}$
denote the unit sphere in $\mathbb{R}^d $ and let $\sigma$ be the
normalized surface measure in $S^{d-1}$. As explained in the
Introduction, we consider
\begin{equation}\label{eq6.1}
{\mathbf{ A}}^2 (f) (x,h) = \int_{S^{d-1}} A_{\xi}^2 (f) (x,h) d
\sigma (\xi) \quad , x \in \mathbb{R}^d , 0<h<1
\end{equation}
Denote by
$e(z,x) = (z-x) / \|z- x \|$. An easy calculation shows
$$
{\mathbf {A}}^2 (f) (x,h)= \int_{\Gamma (x) \cap \{t>h \} }
\left|\frac{f(z+te(z,x)) + f(z-
 te(z,x)) - 2 f(z) }{t} \right|^2 \frac{dm(z)}{\|z-x\|^{d-1}} \frac{dt} { t^2 }
$$
Here $\Gamma (x) = \{ (z, t) \in \mathbb{R}^{d+1}: z \in
\mathbb{R}^d, 0<t < 1, |z-x|< t \}$ and $dm(z)$ denotes Lebesgue
measure in $\mathbb{R}^d$. Consider also the following averaged
version of $\tilde {\Delta}_\xi$. Given a measurable subset $E
\subset S^{d-1}$, consider
$$
\tilde {\Delta} (f) (x,h,E )= \int_E \tilde {\Delta}_\xi (f)(x,h) d
\sigma (\xi) \quad , x \in \mathbb{R}^d, 0 < h <1
$$
An easy calculation shows
$$
\tilde {\Delta} (f) (x,h,E )= \int_{h/2}^h \int_{E(t)} \frac{f(z+
te(z,x)) - f(z- t e(z,x))}{2t} \frac{d m(z)}{\|z-x \|^{d-1}}
\frac{dt}{ 2 t^2}
 $$
where $E(t)= \{z \in \mathbb{R}^d : 0< \|z-x \|< t, e(z,x) \in E
\}$. The rest of this Section is devoted to the proof of Theorem 5.
We start with an elementary  auxiliary result.
\begin{lemma}\label{lemma6.3}
Let $E$ be a measurable set contained in a ball $B \subset
\mathbb{R}^d $. Assume $m (E) >  2 m (B)/ 3$. Then for any point $x
\in B$ there exists $y \in E$ such that $(x+y)/ 2 \in E$.
\end{lemma}

\begin{proof}
Fist consider the one dimensional case $d=1$. One can assume $B=[0,1]$ and $x=0$. Since $|2E \cap [0,1] | = 2 |E \cap [0, 1/2]| > 1/3 $, we deduce that $|E \cap 2E|>0$. So we may pick $y \in E \cap 2E$. In the higher dimensional case $d>1$, observe that given $x \in B$, there exists a line segment $L \subset B$ ending at $x$ such that the length of $L \cap E$ is bigger than $2/3 |L|$. Now the one dimensional result can be  applied to obtain $y \in L \cap E$ such that $(x+y)/ 2 \in E$.
\end{proof}

\begin{proof}[Proof of Lemma 1]
For $N=1, 2, \ldots$, let $E_N$ be the set of points $x \in {\cal
U}$ such that $|f(x+t e_i)  - f(x)| < N |t| $ for any $|t|< 1/N$ and
$i=1,2, \ldots, d$ and moreover $|\Delta_2 (f) (x,h)| < N$ for any
$h \in \mathbb{R}^d $ with $0< |h|< 1/N$. Fix $N = 1, 2, \ldots$ and
let us show that $f$ is differentiable at almost every point of
$E_N$. Let $x$ be a point of density of $E_N$. Pick $\delta >0$ such
that $m (E \cap B(x,t))
> 2 m (B(x,t)) / 3$ for any $0<t<\delta$. Here $B(x,t)$ denotes
the ball centered at $x \in \mathbb{R}^d $ and radius $t>0$. We can
assume that $\delta< 1/2N$. Let $h \in \mathbb{R}^d $ with $|h|<
\delta$. Write $h = \sum_{j=1}^{d} h_j e_j$, $x_0 = x$ and $x_k = x
+ \sum_{j=1}^{k} h_j e_j$ for $k=1, 2, \ldots , d$. Then $f(x+h) -
f(x) = \sum_{k=1}^{d} (f(x_k) - f(x_{k-1}))$. Fix $k=1, 2, \ldots ,
d$. Apply Lemma 6.1 to the point $x_k \in B(x, 2|h|)$ and the set
$E_N \cap B(x , 2 |h|)$ to obtain a point $y_k \in E_N \cap B(x , 2
|h|) $ such that $(x_k + y_k) / 2 \in E_N \cap B(x , 2 |h|)$. Note
that $x_k = x_{k-1} + h_k e_k$. Observe that
\begin{align*}
&\frac{f(x_k) - f (x_{k-1})}{|h|} - \frac{f(y_k + h_k e_k) - f (y_k)}{|h|} =\\
& =\frac{ f(x_k) + f(y_k) -2 f ((x_k + y_k)/2)  }{|h|} - \frac{ f(x_{k-1}) + f(y_k + h_k e_k) -2 f ((x_k + y_k)/2)  }{|h|}
\end{align*}
Since $(x_k + y_k)/ 2 \in E_N$, the second term in the identity
above is bounded by $2N$. Since $y_k \in E_N$, we deduce that
$|f(x_k) - f(x_{k-1})| < 3N |h|$. Adding in $k=1,2, \ldots , d$, one
deduces  $|f(x+h) - f(x)| < 3Nd |h|$. We can now apply Stepanov
Theorem to deduce that $f$ is differentiable at almost every point
of $E_N$.
\end{proof}

We now prove Theorem 5.
\begin{proof}[Proof of Theorem 5]
We start with part (a). For $N=1,2,\ldots$, consider the set $A_N$
of points $x \in {\cal U}$ such that $|f(x+h) - f(x)| < N|h|$ for
any $h \in  \mathbb{R}^d$ with $|h|< 1/N$. Note that every point of
$A$ is in infinitely many $A_N$. Fix $N=1,2, \ldots$ and a bounded
measurable set $E \subset A_N$, we will show that almost every point
of $E$ is in $B$. Fix $\xi \in  \mathbb{R}^d$ with $|\xi| = 1$ and
consider the orthogonal hyperplane $\Pi (\xi)= \{x \in  \mathbb{R}^d
: \langle x, \xi \rangle =0 \}$.  As before, for any $\tilde {x} \in
\Pi (\xi)$ consider the function $f_{\tilde {x}}$ which is defined
on the open one dimensional set $E(\tilde{x}) = \{ s \in \mathbb{R}
: \tilde{x} + s \xi \in E \}$. Let $m_{d-1} $ denote Lebesgue
measure in $\Pi (\xi)$. We have
$$
\int_E A_\xi^2 (f) (x) d m (x) =  \int_{\Pi (\xi)}
\int_{E(\tilde{x})} A^2 (f_{\tilde{x}}) (s) ds dm_{d-1} (\tilde{x})
$$
Since $E \subset A_N$, for any $\tilde{x} \in \Pi (\xi)$, the
function $f_{\tilde{x}}$ is locally Lispchitz at each point of
$E(\tilde{x})$ with constant $N$. The  proof of Theorem 2 gives that
there exists a constant $C>0$ independent of $\tilde {x}$ such that
$$
\int_{E(\tilde{x})} A^2 (f_{\tilde{x}}) (s) ds < CN^2
$$
Hence
$$
\int_E A_\xi^2 (f) (x) d m (x) < C(E) N^2
$$
Integrating on $\xi \in S^{d-1}$, we deduce that
$$
\int_E {\mathbf{ A}}^2 (f) (x) d m (x) < C(E) N^2
$$
and thus ${\mathbf{ A}} (f) (x) < \infty$ at almost every $x \in E$.
Hence almost every point of $E$ is in $B$. This finishes the first
inclusion. To show the converse, for $N=1,2,\ldots$, consider the
set  $B_N$ of points $x \in {\cal U}$ such that ${\mathbf{ A}} (f)
(x) <N$ and $|\Delta_2 (f) (x,h)| < N$ for any $h \in  \mathbb{R}^d$
with $0< |h| < 1/N$. Observe that every point of $B$ is in
infinitely many $B_N$. Fix $N=1,2, \ldots$ and a bounded measurable
set $E \subset B_N$. We will show that $f$ is differentiable at
almost every point of $E$. Since
$$
N^2 m (E) > \int_E {\mathbf{ A}}^2 (f) (x) d m (x) = \int_{S^{d-1}}
\int_E A_\xi^2 (f) (x) d m (x) d \sigma (\xi) \, ,
$$
we deduce that for almost every ($\sigma$)  $\xi \in S^{d-1}$ we have
that $A_\xi (f) (x ) < \infty$ at almost every $x \in E$. Part (a) of Theorem 4 gives that for almost every
 ($\sigma$) $\xi \in S^{d-1}$, the directional derivative $D_\xi (f) (x)$ exists at almost every point $x \in E$. Pick a basis $\{\xi_1 , \xi_2 , \ldots , \xi_d \} \in S^{d-1}$ such that for any $i=1,2,\ldots , d$, the corresponding
 directional derivative $D_{\xi_i} (f) (x)$ exists at almost every $x \in E$. Applying  Lemma 1 one concludes that $f$ is differentiable at almost every point $x \in E$.

The proof of part (b) follows closely the arguments of the proof of
Theorem 1.  Let $f\in L^2_{\mathrm{loc}}(\mathbb{R}^d)$. Fix a
measurable set $E \subset S^{d-1}$ with $\sigma (E) >0$ and consider
the maximal function
$$
{\mathbf{ N}} (f) (x, h)  =  \frac{1}{\sigma (E)\ln 4} \sup_{1 \ge
y\ge h} \left( \tilde {\Delta} (f) (x,y,E )  - \tilde {\Delta}
(f)(x,H(y), E)- \frac{1}{2} {\mathbf{ A}}^2(f)(x,y)\right)
$$
Since both $\tilde {\Delta} (f) (x,y,E )$ and ${\mathbf{ A}}^2(f)(x,y)$ are
 means of their one dimensional analogues, Lemma 3.2 and Jensen's
  inequality give that for any cube $Q \subset  \mathbb{R}^d$ with $m (Q) = 1$ and  any $0<\alpha<1$, $0<h<1$ , one has
$$
\int_Q\exp (\alpha {\mathbf{ N}} (f)(x,h))\,dm (x)
\le\frac{C}{1-\alpha}
$$
where $C>0$ is a universal constant independent of $\alpha$,  $h$, $Q$ and $f$. Now the proof proceeds as the proof of Theorem 1.

\end{proof}

\section{Open Questions}\label{sec7}

In this Section we collect several natural questions closely related to our results.

\noindent {\bf 1}. An easy calculation shows that
$$
\tilde{\Delta}(f)(x,h) = \frac{1}{4h} \int^{x+2h}_{x-2h} \Delta (f)
(s,2h) K((s-x)/h) ds \quad , x \in \mathbb{R} \, , \, 0<h<1\, ,
$$
where $K$ is a function supported in $[-2,2]$, $K \equiv 1$ in $[-1,
1]$ and $K(s)= -1/3 + 4/3w^2$ in $[-2,2] \setminus [-1, 1]$. It is
natural to ask for a result similar to Theorem 1 for different
kernels $K$. Also,  consider
$$
\tilde{\Delta}^* (f)(x,h)= \fint^{x+h}_{x-h} \Delta f(s,h) ds
$$
Our arguments give a Law of the Iterated Logarithm relating the growth of $\tilde{\Delta}^*
(f)(x,2^{-N})$ and a discrete version of $A(f)$ given by
$$
\sum_{k=1}^N  \int_{x-{2^{-k}}}^{x+{2^{-k}}} \Delta_2^2 (f)(s,
2^{-k}) ds
$$

\noindent {\bf 2}. It is natural to ask for a lower bound in the Law
of the Iterated Logarithm given by Theorem~\ref{theor1}. More
concretely, under which conditions on the function $f$ is the
$\limsup$ in Theorem~\ref{theor1} bounded below by a positive
constant? In the context of boundary behavior of
 harmonic functions in an upper half space, such lower bound was
 proved by Ba\~nuelos, Klemes and Moore. See \cite{BKM2} or \cite[p.~75]{BM}

\noindent {\bf 3}. Stein and Zygmund proved that the set of points
where $f$ is differentiable in the $L^2$ sense coincides, up to sets
of Lebesgue measure zero, with the set of points $x\in\mathbb{R}$
for which there exists $\delta = \delta (x)>0$ such that
$$
\int^{\delta}_0 \Delta_2^2 (f) (x,h) \frac{dh}{h}<\infty.
$$
See \cite{SZ2} or \cite[p.~262]{St1}. So, it is natural to ask if
this set also coincides almost everywhere with the set of points~$x$
where $A(f)(x)<\infty$.


\noindent {\bf 4}. As in the classical situation,
Theorem~\ref{theor2} applies to functions defined at every point of
an open set. Let $f$ be a function defined in an open set ${\cal
U}$. Given a set~$E \subset {\cal U}$, it is natural to ask under
which conditions the function $f$ coincides almost everywhere with a
function which is differentiable in $E$. In one variable this was
considered by Neugebauer (\cite{N}) and his description was
expressed in terms of the square function $g(f)$ mentioned in the
Introduction. It is reasonable to expect a similar result with the
square function $A(f)$ instead $g(f)$.

\noindent {\bf 5}. It is reasonable to expect that the set~$A$ in
Theorem~\ref{theor2} also coincides almost everywhere with the set
$$
C=\left\{x\in {\cal U} :\sup_{0<h<h_0 }
\int^h_{h/2}\int_{x-y}^{x+y}\left|\frac{f(s+y)-f(s-y)}{y}\right|^2\,ds\frac{dy}{y^2}<\infty\right\}
$$
but we have not worked the details. It is obvious that $A\subseteq C$ but the converse is not clear and it could happen one has to add a pointwise condition on the symmetric differences.

\noindent {\bf 6}. In relation to Theorem~\ref{theor3}, we mention
that we have not explored analogue descriptions of Sobolev spaces
with higher order derivatives.


\noindent {\bf 7}. We also do not know if Sobolev spaces $W^{1,p}
(\mathbb{R}^d)$ can be described using the square function $\mathbf{
A} (f)$ defined in ~\eqref{eq6.1}. Let $f \in W^{1,p} (\mathbb{R}^d)
$, $1<p<\infty $. Theorem~\ref{theor4} tells that for any $\xi \in
\mathbb{R}^d$, $|\xi| = 1$, one has
 $A_\xi (f) \in L^p (\mathbb{R}^d)$ and $\|A_\xi (f) \|_p < C(p) \| f \|_{W^{1,p} (\mathbb{R}^d)}$. Minkowski
 integral inequality gives that $\| \mathbf{ A}(f)\|_p < C(p) \| f \|_{W^{1,p} (\mathbb{R}^d)}$. The converse seem to require some work and we have not explored it.

\end{document}